\documentclass[12pt]{article}
\usepackage[usenames, dvipsnames]{color}
\usepackage{authblk}
\usepackage{amsfonts}
\usepackage{mathrsfs}
\usepackage{amsmath}
\usepackage{amssymb,extarrows}
\usepackage{multicol}
\usepackage{float}
\usepackage{makeidx}
\usepackage{layout}
\usepackage{array}
\usepackage{a4wide}
\usepackage{boxedminipage}
\usepackage{hyperref}
\usepackage{latexsym}
\usepackage{color}
\usepackage{dsfont}
\usepackage{graphicx}
\usepackage{ulem}
\usepackage{amsthm}
\title{Double periodic viscous flows in infinite space-periodic pipes}
\author[1]{Hugo Beir\~{a}o da Veiga \thanks{Hugo Beir\~{a}o da Veiga (\texttt{hbeiraodaveiga@gmail.com}) is partially supported by FCT (Portugal) under the project: UIDB/MAT/04561/2020.}}
\affil[1]{\small{Department of Mathematics, Pisa University, Pisa, Italy}}
\author[2]{Jiaqi Yang\thanks{Jiaqi Yang (\texttt{yjqmath@nwpu.edu.cn, yjqmath@163.com}) is supported by NSF of China under Grant: 12001429.}}
\affil[2]{\small{School of Mathematics and Statistics, Northwestern Polytechnical University, Xi'an, 710129, China}}
\date{}

\newtheorem{theorem}{Theorem}[section]
\newtheorem{proposition}[theorem]{Proposition}
\newtheorem{lemma}[theorem]{Lemma}

\newtheorem{problem-2}[theorem]{Problem}
\theoremstyle{remark}
\newtheorem{remark}{Remark}[section]

\theoremstyle{definition}

\numberwithin{equation}{section}

\newcommand{\mM}{\mathcal{M}}

\newcommand{\mT}{\mathcal{T}}

\newcommand{\pa}{{\partial}}

\newcommand{\ro}{{\rho}}

\newcommand{\Om}{{\Omega}}

\newcommand{\Si}{{\Sigma}}
\newcommand{\g}{{\gamma}}

\newcommand{\te}{{\theta}}

\newcommand{\ux}{{\underline{x}}}

\newcommand{\ep}{{\epsilon}}

\newcommand{\p}{\partial}
\newcommand{\e}{\epsilon}

\newcommand{\R}{\mathbb{R}}
\newcommand{\N}{\mathbb{N}}

\newcommand{\al}{\alpha}
\newcommand{\f}{\frac}
\newcommand{\n}{\nabla}

\newcommand{\la}{\lambda}
\newcommand{\La}{\Lambda}
\newcommand{\ee}{\mathbf{e}}
\newcommand{\cT}{{\widetilde{\textbf{T}}}}

\newcommand{\ed}{{\end{document}}}
\newcommand{\bn}{\mathbf{n}}
\newcommand{\bw}{\mathbf{w}}

\newcommand{\bu}{\mathbf{u}}
\newcommand{\ba}{\mathbf{a}}
\newcommand{\bb}{\mathbf{b}}
\newcommand{\bv}{\mathbf{v}}
\newcommand{\bph}{\boldsymbol{\phi}}

\begin{document}
\maketitle
\begin{abstract}
	We study the motion of an incompressible fluid in an $n+1$-dimensional infinite pipe $\,\La\,$ with an $L$-periodic shape in the $z=x_{n+1}$ direction. We set $\,x=(x_1,x_2,\cdots,x_{n})$, and $z=x_{n+1}$. We denote by $\Sigma_z$ the cross section of the pipe at the level $z\,,$ and by $v_z$ the $n+1$ component of the velocity. Fluid motion is described by the evolution Stokes or Navier-Stokes equations together with the non-slip boundary condition $\bv=\,0\,$. We look for solutions $\bv(x,z,t)$ with a given, arbitrary, $T-$time periodic total flux $\,\int_{\Sigma_z} \,v_z(x,z,t)\,dx=g(t)\,,$ which should be simultaneously $T$-periodic with respect to time and $L$-periodic with respect to $z\,.$ We prove existence and uniqueness of the solution to the above problems. The results extend those proved in reference \cite{B-05}, where the cross sections were independent of $z$. The argument is presented through a sequence of steps. We start by considering the linear, stationary, $z-$periodic Stokes problem. Then we study the double periodic evolution Stokes equations, which is the heart of the matter. Finally, we end with the extension to the full Navier-Stokes equations.\par%
	
\end{abstract}
%\vspace{0.2cm}
%
\section{Introduction and main results}
We start by remarking that the physical motivations leading to the present paper are similar to those claimed in reference \cite{B-05}. So we recommend the reading of the Introduction of the above reference.\par%
We study the motion of an incompressible fluid in an $n+1$-dimensional infinite pipe $\,\La\,$ with an $L$-periodic shape in the $z=x_{n+1}$ direction. This notation is due to the distinct role played by $x_{n+1}\,.$ Below  $\ux=\,(x_1,x_2,\cdots,x_{n},x_{n+1})$, $\,x=(x_1,x_2,\cdots,x_{n})$, and $z=x_{n+1}$. We denote by $\Sigma_z$ the cross section of the pipe at the level $z\,,$ and by $v_z$ the $n+1$ component of the velocity. Fluid motion is described by the evolution Stokes or Navier-Stokes equations together with the non-slip boundary condition $\bv=\,0\,$. Let $g(t)$ be a given real $T$-periodic function. We look for solutions $\bv(x,z,t)$ with time-periodic total flux $\,\int_{\Sigma_z} \,v_z(x,z,t)\,dx=g(t)\,,$ which are simultaneously $T$-periodic with respect to time and $L$-periodic with respect to $z\,,$ for all $\,z\in \R\,$ and all $\,t\in \R\,.$ Everywhere in the sequel $z-$periodicity means space periodicity in the $z$ direction with the given amplitude $L$.\par%
We prove existence and uniqueness of the solution to the above problems, see Theorems \ref{thm:ex} and \ref{thm:navstokes} below. The results substantially extend those proved in reference \cite{B-05}, where cross sections were independent of $z$. In fact, by applying the results proved below to the above  particular case, one shows that solutions are $z$-periodic, for all period $L>0$. This implies their independence on $z\,$, hence the main result in the above 2005 reference follows. More precisely, the smallness assumption in Theorem \ref{thm:navstokes} below was not required in reference \cite{B-05}. This difference is essentially due to the fact that, in \cite{B-05}, a careful analysis of the equations and hypothesis has shown that the non-linear term vanishes. Hence, in \cite{B-05}, there is no essential distinction between statements for Stokes and Navier-Stokes problems (in \cite{B-05} this distinction comes to light only in treating the Leray's problem).%

\vspace{0.2cm}

The hard core of the paper, subsections \ref{auxprob} and \ref{proofteo}, follows that in reference \cite{B-05}. This should not hide that, to overcome new obstacles, many deep additional arguments have been introduced. This was mainly due to the fact that now the velocity field depends also on $z=\,x_{n+1}\,$ and, furthermore, it is not reduced to a scalar (the $v_z$ component of the velocity), as in \cite{B-05}. Necessary extensions and non trivial new arguments have been developed.%

\vspace{0.2cm}

As usual, we re-elaborate in a suitable way the classical equations which describe the physical problems, to obtain precise, more abstract, mathematical formulations.\par%
We opt to divide the proofs into a sequence of steps. We start by considering the linear, stationary, $z$-space periodic Stokes problem \eqref{stokes-per}, which mathematical formulation is given by equation \eqref{stokes-per-abs}. Taking this case as a reference, we consider the double periodic evolution Stokes equations
\begin{equation}\label{eq-1}
	\begin{cases}
		\f{\p \bv}{\p t}-\nu\Delta \bv+\n p=0& \text{in $\La$}\,,\\
		\n \cdot \bv=0& \text{in $\La$}\,,\\
		\bv=0& \text{on $S$}\,,\\
		\int_{\Sigma_z}v_z\,d\Sigma_z=g(t)\,,\\
		\bv(x,z+L,t)=\bv(x,z,t)\,,\\
		\bv(x,z,t+T)=\bv(x,z,T)\,.
	\end{cases}
\end{equation}
The resolution of this problem is the core of our paper. We will appeal to its more abstract formulation \eqref{eq-1(2)} based on the corresponding abstract formulation of the stationary problem. Successively, equation \eqref{eq-1(2)} will be written, and solved, under the final form \eqref{eq-4}. See Theorem \ref{thm:ex} below. Everywhere the lower symbol $\,\#\,$ means that $T-$time periodicity is assumed.
\begin{theorem}\label{thm:ex}
	Let a $T-$periodic function $g\in H^1_{\#}(\R_t)$ be given. There is a unique solution $\bv \in \,L^2_{\#}(\R_t;\mathbb{V}(\La))\,$ of the double periodic evolution Stokes problem \eqref{eq-4}, the abstract formulation of problem \eqref{eq-1}. Moreover, there is a constant $c$ depending on $C_0$ and $C_1$ (see equations \eqref{C_0} and \eqref{C_1}), such that $\bv$ satisfies the estimates
	\begin{equation}\label{thm:ex-estimate1}
		\|\Delta \bv\|^2_{L^2_{\#}(\R_t;\mathbb{H}(\La))}\leq c\|g\|^2_{L^2_{\#}(\R_t)}+\f{c}{\nu^2}\|g'\|^2_{L^2_{\#}(\R_t)}\,,
	\end{equation}
	\begin{equation}\label{thm:ex-estimate2}
		\|\bv'\|^2_{L^2_{\#}(\R_t;\mathbb{H}(\La))}\leq c\nu^2\|g\|^2_{L^2_{\#}(\R_t)}+c\|g'\|^2_{L^2_{\#}(\R_t)}\,,
	\end{equation}
	and
	\begin{equation}\label{thm:ex-estimate3}
		\|\bv\|^2_{L^2_{\#}(\R_t;\mathbb{V}(\La))}\leq c(1+\nu)\|g\|^2_{L^2_{\#}(\R_t)}+c\left(\f{1}{\nu}+\f{1}{\nu^2}\right)\|g'\|^2_{L^2_{\#}(\R_t)}\,.
	\end{equation}
	In other words, there is a unique solution of Stokes evolution problem \eqref{eq-1} in $\La\,.$ In particular, $\,\bv\,$ satisfies the adherence boundary condition $\bv|_{S}=0\,,$ and also the conditions\\
	(i) $\bv$ is $T$-time periodic,\\
	(ii) $\bv$ is $L$-periodic with respect to $z$,\\
	(iii) The total flux satisfies $\int_{\Sigma_z}v_z\,d\Sigma_z=g(t)$\,.
\end{theorem}
From the above estimates it easily follows that
\begin{equation}\label{str-est}
	\bv \in \,L^2_{\#}(\R_t;\mathbb{V}_2(\La)) \cap\,C_{\#}(\R_t;\mathbb{V}(\La))\,.
\end{equation}
The next step will be to extend the above result under the addition of an external force $\mathbf{f}$, see equations \eqref{eq-1-f}, by proving  the Theorem \ref{thm-v1+v2} below. It is worth noting that we need here quite sharp estimates to allow the next two steps consisting in replacing in \eqref{eq-1-f} $\mathbf{f}$ by  $\,-\bw\cdot\n\bw\,$, see \eqref{eq-NS-L}, and finally in proving the extension to the full Navier-Stokes equations by a contraction's map argument. This leads to the following result.
\begin{theorem}\label{thm:navstokes}
	Let a $T-$periodic function $g\in H^1_{per}(\R_t)$ be given. There is a positive constant $\,c(\nu)\,$ such that if
	\begin{equation}\label{cnug2}
		\|g\|_{H^1_{per}(\R_t)}< \f{1}{4 c^2(\nu)}\,,
	\end{equation}
	then there is a unique solution $\,\bv \in C_{per}(\R_t;\mathbb{V}(\La))\cap L^2_{per}(\R_t;\mathbb{V}_2(\La))$ of the double periodic evolution Navier-Stokes problem
	\begin{equation}\label{eq-NS}
		\begin{cases}
			\f{\p \bv}{\p t}-\nu\Delta \bv+\,\bv\cdot\n\bv+\n p=\,0& \text{in $\La$}\,,\\
			\n \cdot \bv=0& \text{in $\La$}\,,\\
			\bv=0& \text{on $\p\La$}\,,\\
			\int_{\Sigma_z}v_z\,d\Sigma_z=g(t)\,,\\
			\bv(x,z+L,t)=\bv(x,z,t)\,,\\
			\bv(x,z,t+T)=\bv(x,z,T)\,,
		\end{cases}
	\end{equation}
\end{theorem}
It would be straightforward to preserve here the external force, as in \eqref{eq-1-f}.\par%
In a forthcoming paper we will extend to the present situation the study of the classical Leray's problem done in reference \cite{B-05}.\par%
In Section \ref{sys-case} we show that the solution $\,\bv\,$ is a \emph{full-developed} solution. For convenience $n=2\,,$ and we merely consider a rotation pipe (sections $\Si_z$ are circles with variable radius).%

\vspace{0.2cm}

\begin{remark}\label{desig}
	To avoid further untimely interruptions we opt by anticipating a description of some main differences with respect to \cite{B-05}. A necessary appeal to equations not yet developed looks not uncomfortable to the reader.\par%
	First of all in reference \cite{B-05} there was no difference between Stokes and Navier-Stokes problems since it was proved that the non-linear term vanishes. On the other hand, in the present paper, a central rule is played by the Stokes problem. Hence, to compare the results in  \cite{B-05} to those obtained here, we appeal to Stokes problems.\par%
	In \cite{B-05} the velocity field was parallel to the $z-$direction, hence the unknown was the \emph{scalar} field $\,\chi\,,$ the $z-$component of the velocity $\bv\,.$ Furthermore $\,\chi(x_1, x_2,...,x_n)\,$ does not depend on $z$. On the contrary,
	in the present paper, the velocity field $\,\bv(x_1, x_2,...,x_n,z)\,$ is a \emph{vector} field depending also on $z$. In \cite{B-05} the domain of $\chi$ was the $n-$dimensional generical section $\Si\,$ (denoted by $\Om\,$) while, in the present paper, the basic domain of $\bv$ is (roughly speaking) the $\,(n+1)-$dimensional cell $\,\La_{0,L}\,.$%
	
	\vspace{0.2cm}
	
	Further, the argument that led in \cite{B-05} to the pressure $\,p(z,t)=-\psi(t)z\,,$ see \cite{B-05} equation (5), lead here to the expression $\,p(x,z,t)=-\psi(t)z+p_0(t)+\tilde{p}(x,z,t)\,,$ see \eqref{eq-pressure}. By replacing the above $\,p(z,t)=-\psi(t)z\,,$ into equation (4) in reference \cite{B-05} one gets problem (6)+(10) in this last reference where pressure and divergence free assumption are not any longer present. On the contrary, in the present paper, by substituting the above expression of $\,p(x,z,t)$ into \eqref{eq-1} we have obtained equation \eqref{eq-2}, where  pressure and divergence free assumption are still present. This gives rise to substantial differences between the two situations. Below, to eliminate the pressure, we have to appeal to the decomposition \eqref{ortdec} and to the projection operator $\mathbb{P}\,.$ This led to new, non negligible, obstacles.
\end{remark}

\subsection{Some main related references}
In \cite{galdi-rob} the authors give a proof of a main result in \cite{B-05} by introducing in the proof developed in this last reference a significant relationship between flow rate and axial pressure gradient, which depends only on the cross-section. In \cite{B-06}, the main result in \cite{B-05} is extended to slip boundary conditions. In \cite{galdi-gris}, the authors succeed in extending the theory to non-Newtonian (shear-thinning and shear-thickening) fluids.\par%
The Leray's problem considered in \cite{B-05} was thoroughly studied and extended in reference \cite{berselli-rom} for almost periodic flows, a very interesting result, predict in reference \cite{B-05}. We also would like to quote the challenging results obtained in reference \cite{berselli-miloro} concerning exact solutions to the inverse Womersley problem.\par%
Very interesting, related but distinct problems, have been studied in \cite{chipot}, \cite{KKS-17}, and \cite{Ka-85}.\par%
\subsection{Notes on other possible mathematical strategies}
To solve our problem in the full pipe, one could try to start by solving a suitable problem in a fixed cell, for instance, $\La_{0,\,L}\,,$ and then extend this local solution to the infinite pipe $\,\La\,$ simply by appealing to $\,L-periodicy\,$ in the axis direction. Clearly, if the solution in $ \La_{0,\,L}$ glues in a suitable way with its first $L$ translation, which is defined in $ \La_{L,\,2 L}\,,$ then all the sequences of cells will glue well to each other, at any level $\,z=\,m L\,$. However, even if the solution in the closed interval $\,[0,\,L]\,$ is arbitrarily smooth, and its ``boundary values" on $\,\Si_0$ and $\,\Si_L\,$ coincide, the above extension to $\,(0,\,2 L)\,,$ is not in general a solution. One must introduce additional, suitable, gluing assumptions also for first order $z$ derivatives. Due to these additional assumptions, proofs would be much more technical. We avoid this longer and artificial way by going directly to a global approach in the infinite pipe.
\section{The Stokes stationary, space periodic, problem. A variational, abstract, formulation}\label{stokes st}
In this section, we consider the following stationary Stokes $z-$periodic problem in $\La$:
\begin{equation}\label{stokes-per}
	\begin{cases}
		-\Delta \mathbf{v}+\n p =\mathbf{f}& \text{in $\La$}\,,\\
		\n \cdot \mathbf{v}=0& \text{in $\La$}\,,\\
		\mathbf{v}=0& \text{on $\p\La\,$}\,,\\
		p(x,z+L)=\,p(x)\,,\quad \mathbf{v}(x,z+L)=\mathbf{v}(x,z)\,,
	\end{cases}
\end{equation}
where $ \mathbf{f}(x,z+L)=\,\mathbf{f}(x,z)$. It is worth noting that uniqueness follows immediately from the energy inequality.\par%
Below we will write the above system in a more abstract form, see equation \eqref{stokes-per-abs}, and we will solve this problem, see Theorem \ref{teo-stokes}, by following a well-known road. More precisely, the classical Leray's approach to the Stokes and Navier-Stokes equations (improved by many other authors, in particular E. Hopf) applies as well to the above problem. In this sense, we refer in particular to Temam's well known treatise \cite{temam}. It is worth noting that we do not claim any novelty concerning the resolution of the above system. we merely believe that our presentation may help the interested reader not to acquainted to the method.%
\subsection{Notation. The space domain: An infinite periodic pipe}
We mostly use indifferently the same notation in definitions concerning scalar, vector, or tensor fields. Unless stated differently, all fields of the above types, which depend on the axial variable $z\,,$ are assumed to be defined in the full infinite pipe $\La\,,$ and to be $z-$space periodic. In spite of this agreement, $z-$space periodicity will be often explicitly recalled
when we refer to more physical, say classical, formulations. On the contrary, in more  "abstract" formulations, a reference is, in general, avoided.\par%
We assume that the boundary $\,S=:\pa \La\,$ is smooth, for instance of class $C^2\,.$ Any pipe piece of length $L$ is called pipe element or cell. Let $\Sigma_z$ be the orthogonal cross section of the pipe at the level $z$. Clearly we assume that the non empty sets $\Sigma_z$ are connected. For convenience,  the particular cell
\begin{equation}\label{}
	\La_{0,L}=\{(x,z): x=(x_1,\cdots,x_n)\in\Sigma_z,\, z\in(0,L) \}
\end{equation}
will be used to define norms and other quantities. It is worth noting that this role can be played by any cell $\,\La_{a,a+L}\,,$ for $a\in\R$. We normalize the $(n+1)-$dimensional measure of $\La_{0,L}$ by setting $|\La_{0,L}|=1\,,$ and define $S_{0,L}$ as the lateral boundary of $\La_{0,L}\,.$ Let $\ee_z$ be the unit vector in the $z$-direction. Note that $\ee_z$ does not depend on $z$. The pipe itself is the set
\begin{equation}\label{}
	\La=\bigcup_{z\in\mathbb{Z}}(\,\La_{0,L}\cup\Sigma_0\cup\Sigma_L +z L\ee_z\,)\,.
\end{equation}
It is worth noting that there are neither symmetry assumptions on the shape of each single section $\Sigma_z\,,$ nor possible relations between the shapes of the cross sections for distinct values of $z$. Note that the $z$-axis is not necessarily contained inside $\La\,.$ The simplest example is a spring "skeleton". Roughly speaking, the main point is that, everywhere, the pipe moves forward in a fixed direction (the $z-$direction). It would be of interest to consider pipes which do not obey this condition.\par%
To avoid misunderstanding between our notation and well accepted typical notation, we may in some cases use the symbol $ * $ to recall the above time-periodicity property.\par%
We define $\,\phi \cdot\psi= \sum_{1}^{n+1} \phi_i \psi_i\,,$ with obvious modifications in scalar or tensor cases.\par%
In the sequel, we set
$$
L^2_{*}(\La) =: \{\phi:\,\phi \in \,L^2_{loc}(\overline{\La})\,;\quad \phi(x,z+L)=\,\phi(x,z)\,, \forall\, (x,z) \in \La\}\,,
$$
where, for clarity, we again recall $L-$space periodicity.\par%
In $L^2_{*}(\La)$ we define the scalar product
$$
(\phi,\psi)=: \int_{\La_{0,L}} \,\phi(\ux)\cdot \,\psi(\ux) \,d\ux =\, \,\int_{0}^{L} \int_{\Si_z}\,\phi(x,z)\cdot\,\psi(x,z)\,dxdz\,,
$$
and the corresponding norm $\,\|\phi\|\,$ by setting
$$
\|\phi\|^2 =\,\int_{0}^{L} \int_{\Si_z} \,|\phi(x,z)|^2 \,dx dz\,,
$$
where (as everywhere below) $\La_{0,L}$  may be replaced by any $\La_{a,\,a+L}\,, \forall a \in \R\,.$
Analogously, we define
$$
H^1_{*}(\La) =: \{\phi\in L^2_{*}(\La): \n\phi \in  L^2_{*}(\La)\}\,,
$$
and also
$$
H^1_{0,\,*}(\La) =: \{\phi\in H^1_{*}(\La): \phi_{|S}=\,0\}\,,
$$
where the vanishing assumption on the boundary $S=\,\p \La\,$ is in the usual trace sense.\par%

We define scalar product and norm in  $\,H^1_{0,\,*}(\La)\,$ by setting
$$
((\phi,\psi))=\,\int_{\La_{0,L}} \,\n\phi(\ux)\cdot\,\n\psi(\ux) \,d\ux\,,
$$
and
$$
\|\phi\|_1^2=:\,\int_{\La_{0,L}} \,|\n\phi(\ux)|^2 \,d\ux\,.
$$
Note that $\,\|\phi\| \leq\,C\,\|\phi\|_1\,.$\par%
Furthermore, we consider the linear spaces
$$
C^{\infty}_{*}(\La)=\,\{ \phi \in  C^{\infty}(\La):\, \phi(x,z+L)=\,\phi(x,z)\,, \forall \,z\in \R\,\}\,,
$$
and
$$
C^{\infty}_{0,\,*}(\La)=\,\{ \phi \in C^{\infty}_{*}(\La):\, \textrm{supp}\,\phi \subset \La\}\,.%
$$

\vspace{0.2cm}

Next, we pass to the functional spaces specifically related to the Stokes problem. Following a classical way, we define the linear space
$$
\mathcal{V}(\La) =: \{\phi \in C^{\infty}_{0,\,*}(\La):\, \n\cdot \phi=\,0\,\}
$$
and we denote the closure of $\mathcal{V}(\La)$ in $L^2_{*}(\La)\,$  by $\mathbb{H}(\La)\,,$ and the closure of $\mathcal{V}(\La)$ in $H^1_{0,\,*}(\La)\,$  by $\mathbb{V}(\La)$. One has
$$
\mathbb{H}(\La)=\,\{\bu\in L^2_{*}(\La): \, \n\cdot \bu=\,0\,, \quad (\bu \cdot \bn)|_S=\,0\,\}\,,
$$
where $\bn$ denotes the external normal to the boundary $S\,.$ The boundary condition $ \bu \cdot \bn=\,0\,$ holds in a well-known weak sense. Next, we define the space
$$
\mathbb{V}(\La)=\{\bu\in H^1_{0,\,*}(\La): \n\cdot \bu=\,0\}
$$
normed by $\,\|\cdot\|_{\mathbb{V}}=\|\cdot\|_1 \,.$ Note that $\bu|_{S}=0\,.$ \par%
Exactly as in the classical cases we define the space $\mathbb{H}^{\perp}(\La)$ as being the orthogonal complement of $\mathbb{H}(\La)$ in $ L^2_{*}(\La)\,.$ So
\begin{equation}\label{ortdec}
	L^2_{*}(\La)=\,\mathbb{H}\oplus \mathbb{H}^{\perp}(\La)\,.
\end{equation}
Following a classical notation, we denote the related projection by $\mathbb{P}:\, L^2_{*}(\La) \rightarrow \mathbb{H}(\La)\,$. Note that (see, for instance, \cite{dautray-lions}, Chapter XIX, section 1, sub-section 1.4, and references)
\begin{equation}\label{emb}
	\mathbb{V}\subset \mathbb{H} \cong \mathbb{H}' \subset \mathbb{V}'\,,
\end{equation}
where $\,\mathbb{H}\,$ is identified with its dual space.\par%
Let's also introduce the space
$$
\mathbb{V}_2(\La)=\,\mathbb{V}(\La) \cap \,H^2_{loc}(\overline{\La})\,,
$$
where $,H^2_{loc}(\overline{\La})$ may be replaced by $\,H^2(\La_{-\ep,\,L+\ep})\,,$ $\ep>0\,.$
\subsection{The Stokes stationary $z$-space periodic problem. A variational formulation.}

Let's now consider the variational formulation of problem \eqref{stokes-per} followed by us here. Let's explain, in a quite informal way, the approach followed below. One imposes the boundary condition $\,\bu =\,0\,$ with respect to the $\,x\,$ coordinates, and an $L-$periodic assumption with respect to the last coordinate $z\,.$ Roughly speaking, we have a non-slip boundary condition on $x$ and a classical "torus" situation on $z$. The classical approach to each of the two cases easily applies to the present mixed situation, as the reader immediately realizes. In fact, by imitating the argument developed in  \cite{temam}, Chap.I, sec.2, subsec.2.1 (see, in particular, definition 2.1) we show that the problem "find  $\,\bv \in \mathbb{V}\,$ satisfying equation \eqref{stokes-per-abs} below" is a variational formulation of problem \eqref{stokes-per}. On the other hand, the solution of this variational formulation is guaranteed by the Riesz-Fr\'echet representation theorem. Therefore, the following result holds.\par%
\begin{theorem}\label{teo-stokes}
	Given $\mathbf{f} \in L^2(\La)$, or even in $ \mathbb{V}'$, there is a unique solution $\,\bv \in \mathbb{V}\,$ of the problem
	\begin{equation}\label{stokes-per-abs}
		((\bu,\bv))_{\mathbb{V}}=\,(\mathbf{f},\bv)\,, \quad \forall\, \bv \in\,\mathbb{V}\,.
	\end{equation}
	This solution solves the stationary Stokes L-space-periodic problem \eqref{stokes-per}.
\end{theorem}
Following the classical way, we show that the Stokes operator $\,\mathcal{A}:\,\mathbb{V} \rightarrow \,\mathbb{V}'\,,$ defined by
$$
((\bu, \bv))=\,<\mathcal{A} \bu, \bv >_{ \mathbb{V}', \mathbb{V}}, \quad \forall\, \bv \in \mathbb{V}\,,
$$
is an isomorphism.\par%

Let's now consider $\mathcal{A}$ as an operator in $\mathbb{H}$. We will use the notation $\mathcal{A}_H\,.$  We restrict the operator to the domain
$$
D(\mathcal{A}_H)=\{\bv\in\mathbb{V}:\mathcal{A}\bv\in\mathbb{H}\}\,.
$$
By normalizing the linear space $\,D(\mathcal{A}_H)\,$ with the quantity $\,\|\bu\|_{D(\mathcal{A}_H)}=\,\|\mathcal{A} \bu\|\,$ it easily follows that
$$
\mathcal{A}_H: D(\mathcal{A}_H) \rightarrow \mathbb{H}
$$
is an isomorphism.\par%
Clearly $\mathcal{V} \subset D(\mathcal{A}_H)\,,$ so $ D(\mathcal{A}_H)$ is dense in $\mathbb{H}$ (actually, $\mathcal{A}_H\,$ is a self-adjoint, accretive operator, generator of a semigroup)\,.\par%
Let's show that, acting on the above restricted domain $\,D(\mathcal{A}_H)\,$, one has
\begin{equation}\label{APD}
	\mathcal{A}_H=-\mathbb{P}\Delta\,.
\end{equation}
We appeal to an abbreviate but clear notation. Let us assume that $\mathcal{A} \bv=\,\mathbf{f} \in  \mathbb{H}\,.$ Then $\int \bv \cdot \mathbf{f} =\, \int \n \bv \cdot \n \bu \,,$ for each $\bv \in \mathcal{V}$. Hence, $\int \bv \cdot (\mathbf{f} +\, \Delta \bu )=\,0\,,$ for each $\bv \in \mathcal{V}$. It follows that
$\,\mathbf{f} +\, \Delta \bu \in \mathbb{H}^{\perp} $, equivalently $\, \mathbb{P} (\mathbf{f} +\, \Delta \bu)=\,0\,.$ Since $\, \mathbb{P} \mathbf{f} =\,\mathbf{f}\,,$ it follows that $\mathbf{f}=\, -\mathbb{P} \Delta \bu\,.$ This shows \eqref{APD}.\par%
Note that $\,\mathbf{f} +\, \Delta \bu \in \mathbb{H}^{\perp} $ means that there is $p$ such that $\,\mathbf{f} +\, \Delta \bu =\,\n p\,$ which, together with $\,\n \cdot \bu=\,0\,,$ the non-slip boundary condition on $x$, and periodicity on $z$, shows that $\,\bu \in \mathbb{V}_2\,,$ plus the canonical estimates. The proof of $H^2$ regularity of $\,\bu\,$  follows, since the periodic $z-$direction is un-influent.%
\section{The double periodic evolution Stokes problem. The main result.}\label{sub122}
In this section, we consider the double periodic evolution Stokes problem \eqref{eq-1}.

\subsection{The abstract formulation.}\label{}
We start by showing that the following structure of the pressure is \emph{necessary} for the solvability of Problem \eqref{eq-1}. It is interesting to compare with the simpler situation in \cite{B-05}, recall Remark \ref{desig}.  This new, more intricate situation, gave rise to additional obstacles, which required new devices.
\begin{lemma}\label{lemp}
	If the problem \eqref{eq-1} is solvable, then necessarily the pressure has the form
	\begin{equation}\label{eq-pressure}
		p(x,z,t)=-\psi(t)z+p_0(t)+\tilde{p}(x,z,t)\,,
	\end{equation}
	where $p_0(t)$ is an arbitrary function, and $\tilde{p}(x,z,t)$ is a $z$-periodic function. Decomposition \eqref{eq-pressure} is unique up to the arbitrary function $\,p_0(t)\,.$
\end{lemma}
\begin{proof}
	The time variable has no rule in the above decomposition. It is clearly sufficient to prove that, if the first order partial derivatives of a given function $p(x,z)$ are $z$-space periodic, then the following decomposition holds:
	\begin{equation}\label{}
		p(x,z) = - b z  + a  + \tilde p (x,z),
	\end{equation}
	where $a$  and  $b$ are constants, and  $\tilde p$  is $L$-periodic with respect to $z$. Decomposition \eqref{eq-pressure} is unique up to the arbitrary constant $\,a\,.$
	Set
	\begin{equation}\label{}
		a_0(x)=\f{1}{L}\int_0^L(\p_zp)(x,z)dz\,,
	\end{equation}
	we decompose $\p_zp$ as
	\begin{equation}\label{pzp}
		(\p_zp)(x,z)=a_0(x)+\left((\p_zp)(x,z)-a_0(x)\right):=a_0(x)+p_1(x,z)\,.
	\end{equation}
	It is easy to check that
	\begin{equation}\label{p10}
		\int_0^Lp_1(x,\tilde{z})\,d\tilde{z}=0\,.
	\end{equation}
	It follows from \eqref{pzp} that
	\begin{equation}\label{pzp-p}
		p(x,z)=p(x,0)+a_0(x)z+\int_0^zp_1(x,\tilde{z})\,d\tilde{z}\,,
	\end{equation}
	It is worth noting that $\int_0^zp_1(x,y,\tilde{z})\,d\tilde{z}$ is periodic in the $z$-direction due to \eqref{p10}. Hence,
	\begin{equation}\label{}
		\tilde{p}(x,z)=p(x,0)+\int_0^zp_1(x,\tilde{z})\,d\tilde{z}
	\end{equation}
	is periodic in the $z$-direction, and
	\begin{equation}\label{}
		p(x,z)=a_0(x)z+\tilde{p}(x,z)\,.
	\end{equation}
	Finally, since $\p_ip=(\p_ia_0(x))z+(\p_i\tilde{p})(x,z)$ ($i=1\,,\cdots\,,n$) are periodic with respect to $z$, we get $a_0(x)=\text{constant}:=-b$ since $\p_ia_0(x)=0$ ($i=1\,,\cdots\,,n$) must be zero. Thus, we have
	\begin{equation}\label{}
		p(x,z)=-bz+\tilde{p}(x,z)\,.
	\end{equation}
	Uniqueness, up to the constant $\,a\,,$ is obvious.
\end{proof}
Substituting \eqref{eq-pressure} into \eqref{eq-1} we get the following formulation of this last problem.
\begin{equation}\label{eq-2}
	\begin{cases}
		\f{\p \bv}{\p t}-\nu\Delta \bv+\n \tilde{p}=\psi(t)\ee_z& \text{in $\La$}\,,\\
		\n \cdot \bv=0& \text{in $\La$}\,,\\
		\bv=0& \text{on $S$}\,,\\
		\int_{\Sigma_z}v_z\,d\Sigma_z=g(t)\,,\\
		\bv(x,z+L,t)=\bv(x,z,t)\,,\\
		\bv(x,z,t+T)=\bv(x,z,T)\,,
	\end{cases}
\end{equation}
where $\ee_z=\n z\,$ denotes the unit vector in the $z$-direction. Note that, due to the periodic dependence of $\Sigma_z$ on $z$, the same holds for $\ee_z$. Furthermore, due to \eqref{eq-2}$_1\,,$ $\n \tilde{p}-\psi(t)\ee_z$ must be $T-$time periodic. In the sequel each of these quantities will obey this property.\par%
At this point it is interesting to compare \eqref{eq-2} with the corresponding equations (6)+(1) in \cite{B-05} (where $ \ee_z$ was denoted by $\,\ee$\,).\par%
By appealing to \eqref{eq-pressure} and to the results described in the above sections, we write the system \eqref{eq-1} in the equivalent form:
\begin{equation}\label{eq-1(2)}
	\begin{cases}
		\f{d \bv}{dt}+\nu \mathcal{A}_H \bv=\,\psi(t)\mathbb{P}\ee_z\,,\\
		\int_{\Sigma_z}v_z\,d\Sigma_z=g(t)\,,\\
		\bv(x,z,t)=\bv(x,z,t+T)\,,\forall\, t\in \R\,.
	\end{cases}
\end{equation}
$L$-space periodicity is implicit here. We look for solutions which can satisfy $\,\bv(t) \in \mathbb{V}\,,\forall\, t\in \R\,.$\par%
By multiplying both sides by $\mathbb{P}\ee_z$, and by integrating the above equation in $\La_{0,L}$, we show that
\begin{equation}\label{}
	\begin{split}
		\psi(t)\|\mathbb{P}\ee_z\|^2=&\f{d}{dt}\int_{\La_{0,L}}\bv\cdot \mathbb{P}\ee_z \,d\ux+\nu\left(\int_{\La_{0,L}}\mathcal{A}_H \bv\cdot \mathbb{P}\ee_z\,d\ux\right)\\
		=&\f{d}{dt}\int_{\La_{0,L}}\bv\cdot \ee_z \,d\ux+\nu\left(\int_{\La_{0,L}}\mathcal{A}_H \bv\cdot \mathbb{P}\ee_z\,d\ux\right)\\
		=&Lg'(t)+\nu\left(\int_{\La_{0,L}}\mathcal{A}_H \bv\cdot \mathbb{P}\ee_z\,d\ux\right)\,,
	\end{split}
\end{equation}
where we have used that
\begin{equation}\label{}
	\int_{\La_{0,L}}\bv\cdot \mathbb{P}\ee_z \,d\ux=\int_{\La_{0,L}}\bv\cdot \ee_z \,d\ux\,,
\end{equation}
since $\,\mathbb{P}\bv=\,\bv\,.$\par%
Next we set $\,\ee=\f{\mathbb{P}\ee_z}{\|\mathbb{P}\ee_z\|}\,.$ Concerning notation, the reader should not confuse $\ee_z$ with the $z-$component of $\ee\,,$ which do not play any role here.\par%
By appealing to $\ee$ the problem \eqref{eq-1(2)} can be formulated as follows:
\begin{equation}\label{eq-4-}
	\begin{cases}
		\f{d\bv}{dt}+\nu \mathcal{A}_H \bv-\nu\left(\int_{\La_{0,L}}\mathcal{A}_H \bv\cdot \ee\,dxdz\right)\ee=\f{L g'(t)}{\|\mathbb{P}\ee_z\|}\ee\,,\\
		\int_{\Sigma_z}v_z\,d\Sigma_z=g(t)\,,\\
		\bv(x,z,t)=\bv(x,z,t+T)\,.
	\end{cases}
\end{equation}
Hence we need to solve the $T-$periodic system (\cite{B-05}, equations (15) and (16)):
\begin{equation}\label{eq-4}
	\begin{cases}
		\f{d\bv}{dt}+\nu \mathcal{A}_H \bv-\nu(\mathcal{A}_H\bv\,, \ee) \ee=\f{L}{\|\mathbb{P}\ee_z\|}g'(t)\ee\,,\\
		\int_{\Sigma_z}v_z\,d\Sigma_z=g(t)\,,\\
		\bv(x,z,t)=\bv(x,z,t+T)\,,
	\end{cases}
\end{equation}
for $\,t\in \R\,.$ Note that $\|\mathbb{P}\ee_z\|\,$ is a constant. We look for solutions $\,\bv\,$ such that $\,\bv(t) \in\, D(\mathcal{A}_H)=\, \mathbb{V}_2(\La)\,$ for a.e. $\,t\in \R\,.$\par%
The following result is crucial, as explained in Remark \ref{cruciale} below.
\begin{proposition}\label{propcrucial}
	One has $\,\mathbb{P}\ee_z \neq 0.$ More precisely
	\begin{equation}\label{notin}
		\mathbb{P}\ee_z\notin\mathbb{V}(\La).
	\end{equation}
	Clearly, the same holds to $\,\ee\,.$
\end{proposition}
\begin{proof}
	We start by proving that  $\mathbb{P}\ee_z\neq\mathbf{0}$.\par%
	Let $\bv$ be an arbitrary element of $\mathcal{V}(\La)\,.$ Recall that this space is contained (even dense) in $\,\mathbb{H}(\La)\,.$ By an integration by parts, and by taking into account that  $\n\cdot \bv=\,0\,,$ and that $\,(\bv \cdot \bn)|_S=\,0\,,$ where $\bn$ denotes the external normal to the boundary of the full cell $\La_{0,L}\,,$ we show that
	$$
	(\n z,\,\bv)=\,L\,\int_{\Si_z}\,\bv \cdot \bn\, dx\,.
	$$
	Since the above integral does not vanish for all $\,\bv \in \mathcal{V}(\La)$ it follows that $\ee_z\notin \mathbb{H}^{\perp}(\La)\,.$ Hence $\mathbb{P}\ee_z\neq\mathbf{0}$. We have used that $z=0$ on $\Si_0\,,$ but the argument works on any cell.\par%
	Next we prove that $\mathbb{P}\ee_z\notin\mathbb{V}(\La)$. Actually, if $\mathbb{P}\ee_z\in\mathbb{V}(\La)$, then by integrating by parts, we can obtain that
	\[\int_{\La_{0,L}}|\text{div}(\mathbb{P}\ee_z)|^2d\ux+\int_{\La_{0,L}}|\text{curl}(\mathbb{P}\ee_z)|^2d\ux=\int_{\La_{0,L}}|\n\mathbb{P}\ee_z|^2d\ux\,.\]
	Note that $\text{div}(\mathbb{P}\ee_z)=\text{curl}(\mathbb{P}\ee_z)=0$, we have that $\int_{\La_{0,L}}|\n\mathbb{P}\ee_z|^2d\ux=0$, which implies that $\mathbb{P}\ee_z=\mathbf{0}$ since $\mathbb{P}\ee_z\in\mathbb{V}(\La)$ ($\mathbb{P}\ee_z|_{S}=0$). Thus, we have obtain a contradiction since we have proved that $\mathbb{P}\ee_z\neq\mathbf{0}$. Hence $\mathbb{P}\ee_z\notin\mathbb{V}(\La)$.
\end{proof}
\begin{remark}\label{rem:aas}
	For the interested reader, we repeat here a clarifying remark done in reference \cite{B-05}, page 308, which explains why a simpler way looks not suitable to us. Scalar multiplication in $ \mathbb{H}$ of both sides of equation \eqref{eq-4} by $ \mathcal{A}_H\,\bv,$ followed by integration by parts in $\La_{0,L}\,,$ does not give a sharp estimate in terms of the $\,\mathbb{V}-$norm, due to loss of coercivity. In fact the above procedure leads to the estimate
	\begin{equation}\label{vinte}
		\f12 \f{d}{dt}\|\bv\|^2_{\mathbb{V}}+\,\nu \| \mathcal{A}_H \bv\|^2 -\nu (\mathcal{A}_H \bv,\,\ee)^2=\,\f{L}{\|\mathbb{P}\ee_z\|}  g'(t) (\ee,\mathcal{A}_H \bv )\,,
	\end{equation}
	which corresponds to the estimate (20) in ref. \cite{B-05} (a misprint is corrected in reference \cite{B-05-b}). Note that
	$$
	(\mathcal{A}_H \bv,\,\ee)^2 \leq\,\|\mathcal{A}_H \bv\|^2\,.
	$$
	However coercivity fails since $\,(\mathcal{A}_H \bv,\,\ee)^2=\,\|\mathcal{A}_H \bv\|^2\,,$ for $\,\bv=\,\bw\,,$ where (\cite{B-05}, eq.(14)) $\bw \in D(\mathcal{A}_H)\,,$ is the solution to the equation $\mathcal{A}_H \bw=\,\ee\,.$ More precisely, $\,\nu \mathcal{A}_H \bv - \nu (\mathcal{A}_H \bv,\,\ee) \ee=\,0\,,$ since $\mathcal{A}_H \bw=\,\ee\,.$ This fact looks related to another possibly negative situation formulated below in the Remark \ref{cruciale}.\par%
	Note that even less advisable would be multiplication by $ \bv $ instead of $\mathcal{A}_H \bv$, looking for a coercive estimate in $\,\mathbb{H}\,.$ See an explanation on the last rows in page 307, \cite{B-05}.
\end{remark}

In the next sections, we will prove the main Stokes evolution result, namely Theorem \ref{thm:ex}.
\subsection{An auxiliary problem}\label{auxprob}
Let's define $\mathbf{w}\in D(\mathcal{A}_H)$ as the unique solution of the equation (\cite{B-05},(14))
\begin{equation}\label{}
	\mathcal{A}_H\mathbf{w}=\mathbf{e}\,.
\end{equation}
Furthermore, let's set
\begin{equation}\label{C_1}
	C_1^2=(\mathcal{A}_H\mathbf{w}\,,\mathbf{w})=(\n \mathbf{w}\,,\n \mathbf{w}):=((\bw,\bw))\,,
\end{equation}
and also
\begin{equation}\label{C_0}
	C_0^2=\|\mathbf{w}\|^2\,.
\end{equation}
To solve the system \eqref{eq-4}, we first study the system
\begin{equation}\label{eq-aux-prob}
	\begin{cases}
		\f{2\pi k}{T} \bv+\nu\mathcal{A}_H \bu-\nu(\mathcal{A}_H\bu,\ee)\ee=\f{2\pi\,k}{T}\f{L}{\|\mathbb{P}\ee_z\|} q\, \ee\,,\\
		-\f{2\pi k}{T} \bu+\nu\mathcal{A}_H \bv-\nu(\mathcal{A}_H\bv,\ee)\ee=-\f{2\pi k}{T}\f{L}{\|\mathbb{P}\ee_z\|} p\, \ee\,,
	\end{cases}
\end{equation}
where $k\geq1$, and $p$ and $q$ are given reals. This system corresponds to the system (28) in \cite{B-05}. To compare results, the reader should note that in \cite{B-05} (see the Remark 3 in this reference) it was assumed that $\frac{2 \pi}{T}=\,1\,.$\par%
In this section, we prove the following result:
\begin{theorem}\label{thm:aux}
	Problem \eqref{eq-aux-prob} has one and only one solution $(\bu\,,\bv)\in D(\mathcal{A}_H)\times D(\mathcal{A}_H)$. Moreover,
	\begin{equation}\label{eq-aux-prob-estimate}
		\|\mathcal{A}_H\bu\|^2+\|\mathcal{A}_H\bv\|^2\leq \tilde{C}\left(1+\left(\f{2\pi L}{T\nu\|\mathbb{P}\ee_z\|}\right)^2 k^2\right)\left(p^2+q^2\right)\,.
	\end{equation}
	where $\tilde{C}$ depends only on $C_0$ and $C_1$.
\end{theorem}
\begin{proof}
	We follow the proof of Theorem 3 in reference \cite{B-05}. Since $\,\mathcal{A}_H^{-1}\,$ is compact we find an increasing sequence of strictly positive, real, eigenvalues $\la_j$ of $\,\mathcal{A}_H\,,$ and corresponding eigenfunctions $\bw_j\in\mathbb{H}(\La)$, $j=1\,,2\,,\cdots\,$, such that
	\begin{equation}\label{}
		\mathcal{A}_H\bw_j=\la_j\bw_j\,.
	\end{equation}
	Furthermore,
	\begin{equation}\label{}
		(\bw_i\,,\bw_j)=\delta_{ij}\,,
	\end{equation}
	\begin{equation}\label{}
		((\bw_i\,,\bw_j))=\delta_{ij}\la_i\la_j\,,
	\end{equation}
	where $((\bw_i\,,\bw_j))$ means that $(\n \bw_i,\n \bw_j)$\,. Compared with \cite{B-05}, we remark that here $\bw_j$ is a vector not a scalar.
	
	We set $V_m=\text{span}\{\bw_1\,,\bw_2\,,\cdots\,,\bw_m\}$ and look for $\bu_m\,,\bv_m\in V_m$ such that
	\begin{equation}\label{eq-aux-prob-app}
		\begin{cases}
			\left(\f{2\pi k}{T} \bv_m+\nu\mathcal{A}_H \bu_m-\nu(\mathcal{A}_H\bu_m,\ee)\ee\,,\bph\right)=\f{2\pi k}{T}\f{L}{\|\mathbb{P}\ee_z\|}q (\ee,\bph)\,,\\
			\left(-\f{2\pi k}{T} \bu_m+\nu\mathcal{A}_H \bv_m-\nu(\mathcal{A}_H\bv_m,\ee)\ee\,,\bph\right)=-\f{2\pi k}{T}\f{L}{\|\mathbb{P}\ee_z\|}q (\ee\,,\bph)\,,
		\end{cases}
	\end{equation}
	for each $\bph\in V_m$. We look for $\bu_m$ and $\bv_m$ of the form
	\begin{equation}\label{}
		\bu_m=\sum_{j=1}^m\al_j\bw_j\,,\quad \bv_m=\sum_{i=1}^m\beta_j\bw_j\,.
	\end{equation}
	Straightforward calculations show that \eqref{eq-aux-prob-app} is equivalent to $2m$ dimensional system (replacing the $\bph$'s by the above $\bw_l$, $l=1\,,\cdots\,,m$)
	\begin{equation}\label{eq-aux-prob-app-equi}
		\begin{cases}
			\f{2\pi k}{T} \beta_l+\nu\sum_{j=1}^m[\delta_{jl}-(\bw_j,\ee)(\ee,\bw_l)]\la_j\al_j=\f{2\pi k}{T}\f{L}{\|\mathbb{P}\ee_z\|}q (\ee,\bw_l)\,,\\
			-\f{2\pi k}{T} \al_l+\nu\sum_{j=1}^m[\delta_{jl}-(\bw_j,\ee)(\ee,\bw_l)]\la_j\beta_j=-\f{2\pi k}{T}\f{L}{\|\mathbb{P}\ee_z\|}p (\ee,\bw_l)\,,
		\end{cases}
	\end{equation}
	where $l$ runs from $1$ to $m$. Equation \eqref{eq-aux-prob-app-equi} corresponds to equation (32) in \cite{B-05}, page 310. This last equation was not correct. See \cite{B-05-b} for the correct expression, and consequent obvious changes to be made in the same page.\par%
	Following \cite{B-05} and \cite{B-05-b}, it is convenient to interpret \eqref{eq-aux-prob-app-equi} as a system on the unknown $2m-$dimensional column vector
	$$
	X=(\la_1\al_1,...,\la_m\al_m,\la_1\beta_1,...,\la_m\beta_m)=:(X_1,X_2).
	$$
	Set $\g_{jl}=\,\delta_{jl}-\,(\bw_j,\ee)(\ee,\bw_l)\,,$ $j,l=\,1,...,m\,,$ and denote by $\,M\,$ the corresponding  $m \times m $ matrix. The $2m\times 2m\,$ matrix of the system  \eqref{eq-aux-prob-app-equi} has the form $\widetilde{\mM}=\,\f{2\pi k}{T} \mM\,,$ where
	\begin{displaymath}
		\mM =
		\left[ \begin{array}{ccc}
			M & K\\
			-K & M
		\end{array} \right]
	\end{displaymath}
	and $\,K=\,k\, \textrm{diag} [\la_1^{-1},...,\la_m^{-1}]\,.$ By taking into account that $\,X^T \mM X=\,X_1^T M X_1 +\,X_2^T M X_2 \,,$ one shows that $\mM$ is positive definite if and only if $\,M\,$ is positive definite. Again for the reader's convenience, we repeat here the proof given in \cite{B-05}, page 310, since this point is crucial to understand a deep point in our work. Let $\bar{\ee}$ denote the orthogonal projection (in $\mathbb{H}$) of $\ee$ onto $V_m\,.$ Then
	$$
	\sum \g_{jl}\, \xi_j\, \xi_l=\, |\xi|^2- (\xi,\bar{\ee})(\bar{\ee}, \xi) \geq (1-\,\|\bar{\ee}\|^2) |\xi|^2\,,
	$$
	for each $\xi \in \R^m\,.$ Since $\ee \notin V_m \,,$ it follows that $\|\bar{\ee}\|<\,1\,$ (note the main rule of Proposition \ref{cruciale}). Hence we have proved that problem \eqref{eq-aux-prob-app-equi} admits one and only one solution in $V_m\times V_m$.\par%
	\begin{remark}\label{cruciale}
		As we have just seen, the strict positivity of $M$ holds since $\ee \notin V_m \,.$ Note that, curiously, this is a non-regularity assumption. It guarantees a suitable coercivity to solve the single $m-$approximating problems, for all finite $m$, which have shown to be sufficient to our purposes. However, if we try to pass to the limit as $m \rightarrow \infty\,$ we could not obtain a suitable estimate since $\,\|\bar{\ee} \|$ converges to $1$ as $m$ goes to infinity. This looks related to the negative situation described in Remark \ref{rem:aas}.\par%
	\end{remark}
	Let's turn to the proof. Again by following \cite{B-05}, by multiplying the first $m$ equations \eqref{eq-aux-prob-app-equi} by $\la_l\al_l$, the last $m$ equations by $\la_l\beta_l$, and by summing up for $l=1\,,\cdots\,,m$ we obtain (\cite{B-05}, (33))
	\begin{equation}\label{eq-aux-prob-app-equi-sum}
		\begin{split}
			&\nu\sum_{j,l=1}^m[\delta_{jl}-(\bw_j,\ee)(\ee,\bw_l)]\left((\la_j\al_j)(\la_l\al_l)+(\la_j\beta_j)(\la_l\beta_l)\right)\\
			&=\f{2L\pi k}{T\|\mathbb{P}\ee_z\|}\sum_{l=1}^m\la_l(\ee,\bw_l)(q\al_l-p\beta_l)\,.
		\end{split}
	\end{equation}
	Equation \eqref{eq-aux-prob-app-equi-sum} can be written in the equivalent form
	\begin{equation}\label{}
		\begin{split}
			&\nu\|\mathcal{A}_H\bu_m\|^2+\nu\|\mathcal{A}_H\bv_m\|^2-\nu[(\mathcal{A}_H\bu_m\,,\ee)^2+(\mathcal{A}_H\bv_m\,,\ee)^2]\\
			=&\f{2L\pi k}{T\|\mathbb{P}\ee_z\|}[q(\mathcal{A}_H\bu_m\,,\ee)-p(\mathcal{A}_H\bv_m\,,\ee)]\,.
		\end{split}
	\end{equation}
	Hence, we have
	\begin{equation}\label{eq-Aum-Avm}
		\|\mathcal{A}_H\bu_m\|^2+\|\mathcal{A}_H\bv_m\|^2\leq \left(\f{L\pi k}{T\nu\|\mathbb{P}\ee_z\|}\right)^2(p^2+q^2)+2[(\mathcal{A}_H\bu_m\,,\ee)^2+(\mathcal{A}_H\bv_m\,,\ee)^2]\,.
	\end{equation}
	On the other hand, for each $\bph\in V_m$, we have
	\begin{equation}\label{}
		(\mathcal{A}_H\bph-(\mathcal{A}_H\bph,\ee)\ee\,, \bw)=(\bph\,,\ee)-C_1^2(\mathcal{A}_H\bph\,,\ee)\,,
	\end{equation}
	and
	\begin{equation}\label{}
		\|\mathcal{A}_H\bph-(\mathcal{A}_H\bph,\ee)\ee\|^2=\|\mathcal{A}_H\bph\|^2-(\mathcal{A}_H\bph,\ee)^2\,.
	\end{equation}
	Consequently,
	\begin{equation}\label{}
		C^4_1(\mathcal{A}_H\bph\,,\ee)^2\leq 2(\bph\,,\ee)^2+2C_0^2[\|\mathcal{A}_H\bph\|^2-(\mathcal{A}_H\bph\,,\ee)^2]\,.
	\end{equation}
	Thus, we obtain that
	\begin{equation}\label{eq-C14}
		\begin{split}
			&C_1^4[(\mathcal{A}_H\bu_m\,,\ee)^2+(\mathcal{A}_H\bv_m\,,\ee)^2]\\
			\leq& 2[(\bu_m\,,\ee)^2+(\bv_m\,,\ee)^2]+\f{4 C^2_0 L\pi k}{T\nu\|\mathbb{P}\ee_z\|}[q(\mathcal{A}_H\bu_m\,,\ee)-p(\mathcal{A}_H\bv_m\,,\ee)]\,.
		\end{split}
	\end{equation}
	Now, we turn back to the system \eqref{eq-aux-prob-app}. By setting $\bph=\bar{\ee}\,,$ where $\bar{\ee}$ was still defined above, straightforward calculations show that
	\begin{equation}\label{eq-vm-um-e3}
		\begin{cases}
			(\bv_m\,,\ee)=q\|\bar{\ee}\|^2-T\nu\|\mathbb{P}\ee_z\|\f{1-\|\bar{\ee}\|^2}{2L\pi k}(\mathcal{A}_H\bu_m\,,\ee)\,,\\
			(\bu_m\,,\ee)=p\|\bar{\ee}\|^2+T\nu\|\mathbb{P}\ee_z\|\f{1-\|\bar{\ee}\|^2}{2L\pi k}(\mathcal{A}_H\bv_m\,,\ee)\,.
		\end{cases}
	\end{equation}
	From \eqref{eq-C14} and \eqref{eq-vm-um-e3} it follows that
	\begin{equation}\label{}
		\begin{split}
			&\left[C_1^4-4\left(T\nu\|\mathbb{P}\ee_z\|\f{1-\|\bar{\ee}\|^2}{2L\pi}\right)^2\frac{1}{k^2}\right][(\mathcal{A}_H\bu_m\,,\ee)^2+(\mathcal{A}_H\bv_m\,,\ee)^2]\\
			&\leq 4(p^2+q^2)+C^2_0\left\{\left(\f{2L\pi}{T\nu\|\mathbb{P}\ee_z\|}\right)^2 \frac{k^2}{\e}(p^2+q^2)+\e [(\mathcal{A}_H\bu_m\,,\ee)^2+(\mathcal{A}_H\bv_m\,,\ee)^2]\right\}
		\end{split}
	\end{equation}
	for each positive real $\e$. By setting $\e=\f{C^4_1}{4C^2_0}$, letting $m$ be sufficiently large, since $\|\bar{\ee}\|$ converges to $1$ as $m$ goes to $\infty$, we show that
	\begin{equation}\label{}
		\begin{split}
			C_1^4[(\mathcal{A}_H\bu_m\,,\ee)^2+(\mathcal{A}_H\bv_m\,,\ee)^2]
			\leq 16\left[1+\left(\f{C_0}{C_1}\right)^2\left(\f{2L\pi}{T\|\mathbb{P}\ee_z\|}\right)^2\left(\frac{k}{\nu}\right)^2 \right](p^2+q^2)\,.
		\end{split}
	\end{equation}
	Thanks to this estimate, together with \eqref{eq-Aum-Avm}, we get the estimate \eqref{eq-aux-prob-estimate}. From this estimate, the weak convergence in $D(\mathcal{A}_H)\times D(\mathcal{A}_H)$ of the pair $(\bu_m\,,\bv_m)$ to a solution $(\bu\,,\bv)$ of \eqref{eq-aux-prob} follows.
\end{proof}
\subsection{Proof of Theorem \ref{thm:ex}}\label{proofteo}
Following section 5 in \cite{B-05}, we look for solutions $\bv\in L^2_{\#}(\R_t;D(\mathcal{A}_H))$ of the problem \eqref{eq-4} in the form
\begin{equation}\label{}
	\bv(t)=\ba_0+\sum_{k=1}^{\infty}\ba_k\cos \f{2\pi kt}{T}+\sum_{k=1}^{\infty}\bb_k\sin \f{2\pi kt}{T}\,,
\end{equation}
where the unknowns $\ba_k$ and $\bb_k$ belong to $D(\mathcal{A}_H)$.

The data $g\in L^2_{\#}(\R_t)$ is written in the form
\begin{equation}\label{}
	g(t)=p_0+\sum_{k=1}^{\infty}p_k\cos \f{2\pi kt}{T}+\sum_{k=1}^{\infty}q_k\sin \f{2\pi kt}{T}\,,
\end{equation}
where the $p$'s and $q$'s are constants.

Substitution in equation \eqref{eq-4} yields
\begin{equation}\label{eq-a0}
	\mathcal{A}_H\ba_0-(\mathcal{A}_H\ba_0\,,\ee)\ee=0\,,
\end{equation}
together with
\begin{equation}\label{eq-ak}
	\begin{cases}
		\f{2\pi k}{T} \bb_k+\nu\mathcal{A}_H \ba_k-\nu(\mathcal{A}_H\ba_k,\ee)\ee=\f{2\pi k}{T}\f{L}{\|\mathbb{P}\ee_z\|}q_k \ee\,,\\
		-\f{2\pi k}{T} \ba_k+\nu\mathcal{A}_H \bb_k-\nu(\mathcal{A}_H\bb_k,\ee)\ee=-\f{2\pi k}{T}\f{L}{\|\mathbb{P}\ee_z\|}p_k \ee\,,
	\end{cases}
\end{equation}
for all integer $k\geq\,1\,.$ Equation \eqref{eq-a0} is equivalent to
\begin{equation}\label{}
	\ba_0=\tilde{c}\bw\,,
\end{equation}
where $\tilde{c}$ is a constant, which will be determined below by \eqref{eq-4}$_2$, i.e., by $\int_{\Sigma_z}v_z\,d\Sigma_z=g(t)\,,$ or by
$$
(\bv(t)\,,\ee)=\f{1}{\|\mathbb{P}\ee_z\|}\int_{\La_{0,L}}v_z\,d\ux=\f{L}{\|\mathbb{P}\ee_z\|}g(t)\,.
$$
Note that each of the systems \eqref{eq-ak}, $k\in\N$, has the form \eqref{eq-aux-prob}. By Theorem \ref{thm:aux} we show that the coefficients $\ba_k$ and $\bb_k$ are uniquely determined. Moreover, we have the estimates
\begin{equation}\label{eq-estimate-ak-bk}
	\|\mathcal{A}_H\ba_k\|^2+\|\mathcal{A}_H\bb_k\|^2\leq \tilde{C}\left(1+\left(\f{2\pi kL}{T\nu\|\mathbb{P}\ee_z\|}\right)^2\right)\left(p_k^2+q_k^2\right)
\end{equation}
for each $k\in\mathbb{N}$. On the other hand,
\begin{equation}\label{}
	\mathcal{A}_H\bv(t)=\tilde{c}\ee+\sum_{k=1}^{\infty}(\mathcal{A}_H\ba_k)\cos \f{2\pi kt}{T}+\sum_{k=1}^{\infty}(\mathcal{A}_H\ba_k)\sin \f{2\pi kt}{T}\,,
\end{equation}
where $\tilde{c}$ will be determined below. Hence,
\begin{equation}\label{}
	\|\bv\|^2_{L^2_{\#}(\mathbb{R}_t;\mathcal{A}_H)}=\int_0^{T}(\mathcal{A}_H\bv(t)\,,\mathcal{A}_H\bv(t))dt=T \tilde{c}^2+\f{T}{2}\sum_{k=1}^{\infty}\left(\|\mathcal{A}_H\ba_k\|^2+\|\mathcal{A}_H\bb_k\|^2\right)\,.
\end{equation}
Furthermore, by \eqref{eq-estimate-ak-bk}, one has
\begin{equation}\label{}
	\|\bv\|^2_{L^2_{\#}(\mathbb{R}_t;\mathcal{A}_H)}\leq T \tilde{c}^2+\f{\tilde{C}T}{2}\sum_{k=1}^{\infty}(p_k^2+q_k^2)+\f{\tilde{C}T}{2}\sum_{k=1}^{\infty}\left(\f{2\pi kL}{T\nu\|\mathbb{P}\ee_z\|}\right)^2(p_k^2+q_k^2)\,.
\end{equation}
Next we determine $\tilde{c}$ by imposing the constraint $\, (\bv(t)\,,\ee)=\f{L}{\|\mathbb{P}\ee_z\|}g(t)\,.$ By multiplying both sides of \eqref{eq-4} by $\ee$ we show that
\begin{equation}\label{}
	\f{d}{dt}\left[(\bv\,,\ee)-\f{L}{\|\mathbb{P}\ee_z\|}g(t)\right]=0\,.
\end{equation}
On the other hand, we have
\begin{equation}\label{}
	(\bv(t)\,,\ee)=\tilde{c}(\bw\,,\ee)+\sum_{k=1}^{\infty}(\ba_k\,,\ee)\cos \f{2\pi kt}{T}+\sum_{k=1}^{\infty}(\bb_k\,,\ee)\sin \f{2\pi kt}{T}\,.
\end{equation}
Hence, we get
\begin{equation}\label{}
	(\ba_k\,,\ee)=\f{L}{\|\mathbb{P}\ee_z\|}p_k\,,\quad (\bb_k\,,\ee)=\f{L}{\|\mathbb{P}\ee_z\|}q_k\,,
\end{equation}
and
\begin{equation}\label{}
	(\bv(t)\,,\ee)=\tilde{c}(\bw\,,\ee)-\f{L}{\|\mathbb{P}\ee_z\|}p_0+\f{L}{\|\mathbb{P}\ee_z\|}g(t)\,.
\end{equation}
To get $(\bv(t)\,,\ee)=\f{L}{\|\mathbb{P}\ee_z\|}g(t)$, we have to impose that $\tilde{c}=\f{L\,p_0}{\|\mathbb{P}\ee_z\|C^2_1}$. Hence, $\ba_0=\f{L\,p_0}{\|\mathbb{P}\ee_z\|C^2_1}\bw$\,.

Finally, we have
\begin{equation}\label{}
	\|\bv\|^2_{L^2_{\#}(\mathbb{R}_t;\mathcal{A}_H)}\leq T \tilde{c}^2+\f{\tilde{C}T}{2}\sum_{k=1}^{\infty}(p_k^2+q_k^2)+\f{\tilde{C}L^2}{\|\mathbb{P}\ee_z\|\nu^2}\sum_{k=1}^{\infty}\|g'\|^2_{L^2_{\#}(\R_t)}\,.
\end{equation}
This proves \eqref{thm:ex-estimate1}. The estimate  \eqref{thm:ex-estimate2} follows from \eqref{thm:ex-estimate1} together with the first equation \eqref{eq-4}. Finally, the estimate \eqref{thm:ex-estimate3} follows from \eqref{thm:ex-estimate1}, \eqref{thm:ex-estimate2}, and \cite[(23),(25)]{B-05}. A few misprints in the proofs of these two last estimates are corrected in \cite{B-05-b}.

\vspace{0.2cm}

Next, we prove the uniqueness of the solution. Assume that $(\bv_1,\psi_1(t))$ and $(\bv_2,\psi_2(t))$ are two solutions of \eqref{eq-1(2)}. Set $\bu=\bv_1-\bv_2$, then $\bu$ satisfies the following equations:
\begin{equation}\label{eq-3-v1-v2}
	\begin{cases}
		\f{d \bu}{dt}+\nu \mathcal{A} \bu=(\psi_1(t)-\psi_2(t))\mathbb{P}\ee_z\,,\\
		\int_{\Sigma_z}u_z\,d\Sigma_z=0\,.
	\end{cases}
\end{equation}
By multiplying both sides of \eqref{eq-3-v1-v2} by $\bu$, and integrating over $\La_{0,L}$, we obtain
\begin{equation}\label{}
	\f{d}{dt}\int_{\La_{0,L}}|\bu|^2d\ux+\nu\int_{\La_{0,L}}|\n\bu|^2d\ux=(\psi_1(t)-\psi_2(t))\int_{\La_{0,L}}\bu\cdot\mathbb{P}\ee_zd\ux\,.
\end{equation}
Note that
\begin{equation}\label{}
	\int_{\La_{0,L}}\bu\cdot\mathbb{P}\ee_zd\ux=\int_{\La_{0,L}}\bu\cdot\ee_zd\ux=\int_0^L\left(\int_{\Sigma_z}u_z\,d\Sigma_z\right)dz=0\,.
\end{equation}
Hence, we have
\begin{equation}\label{pf-uni-1}
	\f{d}{dt}\int_{\La_{0,L}}|\bu|^2d\ux+\nu\int_{\La_{0,L}}|\n\bu|^2d\ux=0\,,
\end{equation}
which gives that
\begin{equation}\label{}
	\nu\int_0^T\int_{\La_{0,L}}|\n\bu|^2d\ux dt=0\,.
\end{equation}
Therefore we have $\bu=0\,$  since $\bu=0$ on $S_L$.
\section{The nonhomogeneous Stokes equations}\label{nonhomog}
In this section, in view of the full Navier-Stokes equations, we study the following nonhomogeneous Stokes equations:
\begin{equation}\label{eq-1-f}
	\begin{cases}
		\f{\p \bv}{\p t}-\nu\Delta \bv+\n p=\mathbf{f}& \text{in $\La$}\,,\\
		\n \cdot \bv=0& \text{in $\La$}\,,\\
		\bv=0& \text{on $S$}\,,\\
		\int_{\Sigma_z}v_z\,d\Sigma_z=g(t)\,,\\
		\bv(x,z+L,t)=\bv(x,z,t)\,,\\
		\bv(x,z,t+T)=\bv(x,z,T)\,,
	\end{cases}
\end{equation}
where $\mathbf{f} \in L^2_{\#}(\R_t;\mathbb{H}(\La))\,.$ By arguing as in Lemma \ref{lemp} we show that $p(x,z,t)=-\psi(t)z+p_0(t)+\tilde{p}(x,z,t)$. So we can write the above system as follows
\begin{equation}\label{eq-2-f}
	\begin{cases}
		\f{\p \bv}{\p t}-\nu\Delta \bv+\n \tilde{p}=\psi(t)\ee_z+\mathbf{f}& \text{in $\La$}\,,\\
		\n \cdot \bv=0& \text{in $\La$}\,,\\
		\bv=0& \text{on $S$}\,,\\
		\int_{\Sigma_z}v_z\,d\Sigma_z=g(t)\,,\\
		\bv(x,z+L,t)=\bv(x,z,t)\,,\\
		\bv(x,z,t+T)=\bv(x,z,T)\,.
	\end{cases}
\end{equation}
We then look for the solution $(\bv,\psi(t),\tilde{p})$ in the form
\begin{equation}\label{}
	(\bv,\psi(t),\tilde{p})=(\bv^1,0,\tilde{p}^1)+(\bv^2,\psi(t),\tilde{p}^2)\,,
\end{equation}
where $(\bv^1,\tilde{p}^1)$ is the solution of the problem
\begin{equation}\label{eq-2-f1}
	\begin{cases}
		\f{\p \bv^1}{\p t}-\nu\Delta \bv^1+\n \tilde{p}^1=\mathbf{f}& \text{in $\La$}\,,\\
		\n \cdot \bv^1=0& \text{in $\La$}\,,\\
		\bv^1=0& \text{on $S$}\,,\\
		\bv^1(x,z+L,t)=\bv^1(x,z,t)\,,\\
		\bv^1(x,z,t+T)=\bv^1(x,z,T)\,,
	\end{cases}
\end{equation}
and $(\bv^2,\psi(t),\tilde{p}^2)$ is the solution of the problem
\begin{equation}\label{eq-2-f2}
	\begin{cases}
		\f{\p \bv^2}{\p t}-\nu\Delta \bv^2+\n \tilde{p}^2=\psi(t)\ee_z&  \text{in $\La$}\,,\\
		\n \cdot \bv^2=0& \text{in $\La$}\,,\\
		\bv^2=0& \text{on $S$}\,,\\
		\int_{\Sigma_z}v^2_z\,d\Sigma_z=\tilde{g}(t)\,,\\
		\bv^2(x,z+L,t)=\bv^2(x,z,t)\,,\\
		\bv^2(x,z,t+T)=\bv^2(x,z,T)\,,
	\end{cases}
\end{equation}
where
$$
\tilde{g}(t)=g(t)-\int_{\Sigma_z}v^1_z\,d\Sigma_z\,,
$$
and $\psi(t)$ is one of the unknowns of the problem. We start by proving the following theorem.
\begin{theorem}\label{thm-v1}
	Assume that $\mathbf{f}\in L^2_{\#}(\R_t;L^2_{*}(\La))$. Then the problem \eqref{eq-2-f1} admits a unique solution $\bv^1\in L^2_{\#}(\R_t;\mathbb{V}(\La))$. Moreover, there is a constant $c$ depending on $C_0$, $C_1$ and $L$, such that
	\begin{equation}\label{e66}
		\begin{split}
			&\|(\bv^1)'\|_{L^2_{\#}(\R_t;\mathbb{H}(\La))}+(\nu^{-1}+\nu^{\f12})\|\bv^1\|_{C_{\#}(\R_t;\mathbb{V}(\La))}+\nu\|\bv^1\|_{L^2_{\#}(\R_t;\mathbb{V}_2(\La))}\\
			&\leq c\|\mathbf{f}\|_{L^2_{\#}(\R_t;L^2_{*}(\La))}\,.
		\end{split}
	\end{equation}
\end{theorem}
\begin{proof}
	We partially appeal to the proof of the estimate (69) in \cite{B-05}. By multiplying both sides of equation \eqref{eq-2-f1}$_1$  by $\bv^1\,,$ followed by integration on $\La_{0,L}\,,$ we get
	\begin{equation}\label{weakf-1}
		\f12\frac{d}{dt}\|\bv^1\|^2+\nu\|\n\bv^1\|^2\leq\|\mathbf{f}\|\|\bv^1\|\,.
	\end{equation}
	By Young's and Poincare's inequalities we have
	\begin{equation}\label{weakf-2}
		\f12\frac{d}{dt}\|\bv^1\|^2+c\nu\|\n\bv^1\|^2\leq \f{c_1}{\nu}\|\mathbf{f}\|^2\,.
	\end{equation}
	By appealing to this last inequality with $\n\bv^1$ simply replaced by $\bv^1\,,$ one obtains
	\[\|\bv^1(t)\|^2\leq e^{-c\nu t}\|\bv^1(0)\|^2+\f{c_1}{\nu}\int_0^te^{-c\nu(t-s)}\|\mathbf{f}(s)\|^2\,ds\,.\]
	It readily follows that the map $\bv^1(0)\to \bv^1(T)$ has a fixed point in the ball $B\subset \mathbb{H}(\La)$ centered at the origin with radius
	\[\rho=\f{c_1}{\nu}\f{\|\mathbf{f}\|^2}{1-\text{exp}(-T c\nu)}\,,\]
	which gives the existence of weak solutions to problem \eqref{eq-2-f1}, see, for instance, \cite{Tartar} page 60.
	It follows from \eqref{weakf-2} that
	\begin{equation}\label{weakf-3}
		\nu\|\bv^1\|_{L^2_{\#}(\R_t;\mathbb{V}(\La))}\leq c\|\mathbf{f}\|\,.
	\end{equation}
	Furthermore, we have
	\begin{equation}\label{weakf-4}
		(\bv^1)'+\nu\mathcal{A}\bv^1=\mathbb{P}\mathbf{f}(t)\,.
	\end{equation}
	Multiplying by $\mathcal{A}\bv^1$ and integrating over $(0,T)$, we can obtain
	\begin{equation}\label{weakf-5}
		\nu\|\mathcal{A}\bv^1\|_{L^2_{\#}(\R_t;\mathbb{H}(\La))}\leq c\|\mathbf{f}\|\,.
	\end{equation}
	Further, from equation \eqref{weakf-4}, one shows that $\,(\bv^1)' \in L^2_{\#}(\R_t;\mathbb{H}(\La))\,,$ plus a corresponding estimate. From these estimates and \cite[Lemma 1 (23)]{B-05}, we can obtain \eqref{e66}. The uniqueness of the solution in the class $L^2_{\#}(\R_t;\mathbb{V}(\La))$ follows by setting $\mathbf{f}=0$, and by following standard techniques.
\end{proof}
Let's show that, from Theorem \ref{thm-v1}, one gets
\begin{equation}\label{es-tg}
	\left\|\int_{\Sigma_z}v^1_z\,d\Sigma_z\right\|_{H^1_{\#}(\R_t)}\leq c\|\mathbf{f}\|_{L^2_{\#}(\R_t;L^2_{*}(\La))}\,.
\end{equation}
In fact, since $\int_{\Sigma_z}v^1_z\,d\Sigma_z$ is independent of $z$, it follows that
\begin{equation}\label{}
	\int_{\Sigma_z}v^1_z\,d\Sigma_z=L^{-1}\int_0^L\int_{\Sigma_z}v^1_z\,d\Sigma_zdz=L^{-1}\int_{\La_{0,L}}v^1_z\,dxdz\,.
\end{equation}
Thus, one has
\begin{equation}\label{e69}
	\begin{split}
		\left\|\int_{\Sigma_z}v^1_z\,d\Sigma_z\right\|^2_{L^2_{\#}(\R_t)}=&L^{-2}\left\|\int_{\La_{0,L}}v^1_z\,dxdz\right\|^2_{L^2_{\#}(\R_t)}\\
		=&L^{-2}\int_0^T\left(\int_{\La_{0,L}}v^1_z\,dxdz\right)^2dt\\
		\leq&L^{-2}\int_0^T\int_{\La_{0,L}}(v^1_z)^2\,dxdzdt\\
		\leq&L^{-2}\|\bv^1\|^2_{L^2_{\#}(\R_t;L^2_{*}(\La))}\,.
	\end{split}
\end{equation}
Recall that we have assumed $|\La_{0,L}|=1$. Similarly,
\begin{equation}\label{e610}
	\begin{split}
		\left\|\int_{\Sigma_z}(v^1_z)'\,d\Sigma_z\right\|^2_{L^2_{\#}(\R_t)}
		\leq L^{-2}\|(\bv^1)'\|^2_{L^2_{\#}(\R_t;L^2_{*}(\La))}\,.
	\end{split}
\end{equation}
Now, from equations \eqref{e69}, \eqref{e610}, and \eqref{e66}, the estimate \eqref{es-tg} follows.\par%
Next we consider Problem \eqref{eq-2-f2}. Note that this problem has exactly the structure of \eqref{eq-2}, which is equivalent to \eqref{eq-4}. Hence Theorem \ref{thm:ex} applies. So it follows that Problem \eqref{eq-2-f2} admits a unique solution $\,\bv^2\,,$ satisfying the estimates
\begin{equation}\label{}
	\|\Delta \bv^2\|^2_{L^2_{\#}(\R_t;\mathbb{H}(\La)})\leq c\|\tilde{g}\|^2_{L^2_{\#}(\R_t)}+\f{c}{\nu^2}\|\tilde{g}'\|^2_{L^2_{\#}(\R_t)}\,,
\end{equation}
\begin{equation}\label{}
	\|(\bv^2)'\|^2_{L^2_{\#}(\R_t;;\mathbb{H}(\La))}\leq c\nu^2\|\tilde{g}\|^2_{L^2_{\#}(\R_t)}+c\|\tilde{g}'\|^2_{L^2_{\#}(\R_t)}\,,
\end{equation}
and
\begin{equation}\label{}
	\begin{split}
		\|\bv^2\|^2_{C_{\#}(\R_t;\mathbb{V}(\La))}\leq& c(1+\nu)\|\tilde{g}\|^2_{L^2_{\#}(\R_t)}+c\left(\f{1}{\nu}+\f{1}{\nu^2}\right)\|\tilde{g}'\|^2_{L^2_{\#}(\R_t)}\,.
	\end{split}
\end{equation}
By \eqref{es-tg}, we have
\begin{equation}\label{}
	\|\tilde{g}\|_{H^1_{\#}(\R_t)}\leq \|g\|_{H^1_{\#}(\R_t)}+\left\|\int_{\Sigma_z}v^1_z\,d\Sigma_z\right\|_{H^1_{\#}(\R_t)}\leq \|g\|_{H^1_{\#}(\R_t)}+c\|\mathbf{f}\|_{L^2_{\#}(\R_t;L^2_{*}(\La))}\,.
\end{equation}
By collecting the above equations, we obtain the following theorem.
%\end{document}
\begin{theorem}\label{thm-v2}
Problem \eqref{eq-2-f2} admits a unique solution $\,\bv^2\,.$ Moreover, $\bv^2$ satisfies the estimates
\begin{equation}\label{thm:ex-estimate1-L}
	\|\Delta \bv^2\|^2_{L^2_{\#}(\R_t;\mathbb{H}(\La))}\leq
	c\|g\|^2_{L^2_{\#}(\R_t)}+\f{c}{\nu^2}\|g'\|^2_{L^2_{\#}(\R_t)}+c(1+\f{1}{\nu^2})\|\mathbf{f}\|^2_{L^2_{\#}(\R_t;L^2_{*}(\La))}\,,
\end{equation}
\begin{equation}\label{thm:ex-estimate2-L}
	\|(\bv^2)'\|^2_{L^2_{\#}(\R_t;\mathbb{H}(\La))}\leq c\nu^2\|g\|^2_{L^2_{\#}(\R_t)}+c\|g'\|^2_{L^2_{\#}(\R_t)}+c(1+\nu^2)\|\mathbf{f}\|^2_{L^2_{\#}(\R_t;L^2_{*}(\La))}\,,
\end{equation}
and
\begin{equation}\label{thm:ex-estimate3-L}
	\begin{split}
		\|\bv^2\|^2_{C_{\#}(\R_t;\mathbb{V}(\La))}\leq & c(1+\nu)\|g\|^2_{L^2_{\#}(\R_t)}+c\left(\f{1}{\nu}+\f{1}{\nu^2}\right)\|g'\|^2_{L^2_{\#}(\R_t)}\\
		&+c(1+\nu)\|\mathbf{f}\|^2_{L^2_{\#}(\R_t;L^2_{*}(\La))}+\,c\left(\f{1}{\nu}+\f{1}{\nu^2}\right) \|\mathbf{f}\|^2_{L^2_{\#}(\R_t;L^2_{*}(\La))}\,,
	\end{split}
\end{equation}
where $c$ is a constant depending on $C_0$ and $C_1$.
\end{theorem}
By appealing to Theorems \ref{thm-v1} and \ref{thm-v2}, we prove the following result.
\begin{theorem}\label{thm-v1+v2}
Problem \eqref{eq-2-f}, or equivalently Problem \eqref{eq-1-f}, admits a unique solution $\,\bv=:\cT \,f\,.$
Moreover, $\bv$ satisfies the estimates:
\begin{equation}\label{thm:ex-estimate1-f}
	\|\Delta \bv\|^2_{L^2_{\#}(\R_t;\mathbb{H}(\La))}\leq c\|g\|^2_{L^2_{\#}(\R_t)}+\f{c}{\nu^2}\|g'\|^2_{L^2_{\#}(\R_t)}+c(1+\f{1}{\nu^2})\|\mathbf{f}\|^2_{L^2_{\#}(\R_t;L^2_{*}(\La))}\,,
\end{equation}
\begin{equation}\label{thm:ex-estimate2-f}
	\|(\bv)'\|^2_{L^2_{\#}(\R_t;\mathbb{H}(\La))}\leq c\nu^2\|g\|^2_{L^2_{\#}(\R_t)}+c\|g'\|^2_{L^2_{\#}(\R_t)}+c(1+\nu^2)\|\mathbf{f}\|^2_{L^2_{\#}(\R_t;L^2_{*}(\La))}\,,
\end{equation}
and
\begin{equation}\label{thm:ex-estimate3-f}
	\begin{split}
		\|\bv\|^2_{C_{\#}(\R_t;\mathbb{V}(\La))}\leq & c(1+\nu)\|g\|^2_{L^2_{\#}(\R_t)}+c\left(\f{1}{\nu}+\f{1}{\nu^2}\right)\|g'\|^2_{L^2_{\#}(\R_t)}\\
		&+c(1+\nu)\|\mathbf{f}\|^2_{L^2_{\#}(\R_t;L^2_{*}(\La))}+c\left(\f{1}{\nu}+\f{1}{\nu^2}\right)\|\mathbf{f}\|^2_{L^2_{\#}(\R_t;L^2_{*}(\La))}\,.
	\end{split}
\end{equation}
\end{theorem}
\section{The global Navier-Stokes double periodic equations. Proof of Theorem \ref{thm:ex}}\label{global}
In this section, we study the three-dimensional Navier-Stokes system \eqref{eq-NS}. To solve this problem we appeal to the auxiliary system:
\begin{equation}\label{eq-NS-L}
\begin{cases}
	\f{\p \bv}{\p t}-\nu\Delta \bv+\n p=\,-\bw\cdot\n\bw& \text{in $\La$}\,,\\
	\n \cdot \bv=0& \text{in $\La$}\,,\\
	\bv=0& \text{on $S$}\,,\\
	\int_{\Sigma_z}v_z\,d\Sigma_z=g(t)\,,\\
	\bv(x,z+L,t)=\bv(x,z,t)\,,\\
	\bv(x,z,t+T)=\bv(x,z,T)\,,
\end{cases}
\end{equation}
where $\bw\in C_{\#}(\R_t;\mathbb{V}(\La))\cap L^2_{\#}(\R_t;\mathbb{V}_2(\La))$.
From Theorem \ref{thm-v1+v2} it follows that
\begin{equation}\label{}
\begin{split}
	&\|\mT(\bw)\|_{C_{\#}(\R_t;\mathbb{V}(\La))}+\|\mT(\bw)\|_{L^2_{\#}(\R_t;\mathbb{V}_2(\La))}\\
	\leq& c(\nu)\|\bw\cdot\n\bw\|_{L^2_{\#}(\R_t;L^2_{*}(\La))}+c(\nu)\|g\|_{H^1_{\#}(\R_t)}\,,
\end{split}
\end{equation}
where $\,\mT(\bw)=\,\cT \,(-\bw\cdot\n\bw)\,$ is the solution of problem \eqref{eq-2-f} with $\,\mathbf{f}=\,-\bw\cdot\n\bw\,$, and $c(\nu)$ is a constant depending on $\nu$, $C_0$, $C_1$, $L$.
Thanks to Gagliardo--Nirenberg interpolation inequality, we get
\begin{equation}\label{}
\begin{split}
	\|\bw\cdot\n\bw\|_{L^2_{*}(\La)}
	&\leq c \|\bw\|_{L^4_{*}(\La)}\|\n\bw\|_{L^4_{*}(\La)}\\
	&\leq c \|\bw\|^{\f14}_{L^2_{*}(\La)}\|\n \bw\|^{\f34}_{L^2_{*}(\La)}\|\n\bw\|^{\f14}_{L^2_{*}(\La)}\|\n^2 \bw\|^{\f34}_{L^2_{*}(\La)}\\
	&\leq c\|\n\bw\|^{\f54}_{L^2_{*}(\La)}\|\n^2 \bw\|^{\f34}_{L^2_{*}(\La)}\,,
\end{split}
\end{equation}
where $c$ is a uniform constant. Hence
\begin{equation}\label{}
\begin{split}
	&\|\bw\cdot\n\bw\|^2_{L^2_{\#}(\R_t;L^2_{*}(\La))}\\
	\leq& c\|\n\bw\|^{\f52}_{L^{10}_{\#}(\R_t;L^2_{*}(\La))}\|\n^2\bw\|^{\f32}_{L^2_{\#}(\R_t;L^2_{*}(\La))}\\
	\leq& c\|\n\bw\|^{2}_{C_{\#}(\R_t;L^2_{*}(\La))}\|\n\bw\|^{\f12}_{L^{2}_{\#}(\R_t;L^2_{*}(\La))}\|\n^2\bw\|^{\f32}_{L^2_{\#}(\R_t;L^2_{*}(\La))}\,.
\end{split}
\end{equation}
Therefore, we have
\begin{equation}\label{es-T}
\begin{split}
	&\|\mT(\bw)\|_{C_{\#}(\R_t;\mathbb{V}(\La))}+\|\mT(\bw)\|_{L^2_{\#}(\R_t;\mathbb{V}_2(\La))}\\
	\leq& c(\nu)\|\n\bw\|_{C_{\#}(\R_t;L^2_{*}(\La))}\|\n\bw\|^{\f14}_{L^{2}_{\#}(\R_t;L^2_{*}(\La))}\|\n^2\bw\|^{\f34}_{L^2_{\#}(\R_t;L^2_{*}(\La))}+c(\nu)\|g\|_{H^1_{\#}(\R_t)}\,.
\end{split}
\end{equation}
Similarly,
\begin{equation}\label{}
\begin{split}
	&\|\mT(\bw_1)-\mT(\bw_2)\|_{L^2_{\#}(\R_t;\mathbb{V}(\La))}+\|\n^2(\mT(\bw_1)-\mT(\bw_2))\|_{L^2_{\#}(\R_t;L^2_{*}(\La))}\\
	\leq& c(\nu)\|\bw_1\cdot\n\bw_1-\bw_2\cdot\n\bw_2\|_{L^2_{\#}(\R_t;L^2_{*}(\La))}\,.
\end{split}
\end{equation}
On the other hand,
\begin{equation}\label{}
\begin{split}
	&\|\bw_1\cdot\n\bw_1-\bw_2\cdot\n\bw_2\|_{L^2_{*}(\La)}\\
	=&\|\bw_1\cdot\n\bw_1-\bw_1\cdot\n\bw_2+\bw_1\cdot\n\bw_2-\bw_2\cdot\n\bw_2\|_{L^2_{*}(\La)}\\
	\leq&\|\bw_1\cdot\n(\bw_2-\bw_1)\|_{L^2_{*}(\La)}+\|(\bw_1-\bw_2)\cdot\n\bw_2\|_{L^2_{*}(\La)}\\
	\leq& c\|\bw_1\|_{L^4_{*}(\La)}\|\n(\bw_1-\bw_2)\|_{L^4_{*}(\La)}+\|\bw_2-\bw_1\|_{L^4_{*}(\La)}\|\n\bw_2\|_{L^4_{*}(\La)}\\
	\leq& c\|\n\bw_1\|_{L^2_{*}(\La)}\|\n(\bw_2-\bw_1)\|^{\f14}_{L^2_{*}(\La)}\|\n^2(\bw_2-\bw_1)\|^{\f34}_{L^2_{*}(\La)}\\
	&+c\|\n(\bw_1-\bw_2)\|_{L^2_{*}(\La)}\|\n\bw_1\|^{\f14}_{L^2_{*}(\La)}\|\n^2\bw_1\|^{\f34}_{L^2_{*}(\La)}\,.
\end{split}
\end{equation}
Therefore,
\begin{equation}\label{}
\begin{split}
	&\|\bw_1\cdot\n\bw_1-\bw_2\cdot\n\bw_2\|^2_{L^2_{\#}(\R_t;L^2_{*}(\La))}\\
	\leq&c\|\n\bw_1\|^{2}_{C_{\#}(\R_t;L^2_{*}(\La))}\|\n(\bw_1-\bw_2)\|^{\f12}_{L^{2}_{\#}(\R_t;L^2_{*}(\La))}\,\|\n^2(\bw_2-\bw_1)\|^{\f32}_{L^2_{\#}(\R_t;L^2_{*}(\La))}\\
	&+c\|\n(\bw_1-\bw_2)\|^{2}_{C_{\#}(\R_t;L^2_{*}(\La))}\|\n\bw_1\|^{\f12}_{L^{2}_{\#}(\R_t;L^2_{*}(\La))}\|\n^2\bw_1\|^{\f32}_{L^2_{\#}(\R_t;L^2_{*}(\La))}\,.
\end{split}
\end{equation}
Hence, we have
\begin{equation}\label{es-T1-T2}
\begin{split}
	&\|\mT(\bw_1)-\mT(\bw_2)\|_{C_{\#}(\R_t;\mathbb{V}(\La))}+\|\n^2(\mT(\bw_1)-\,\mT(\bw_2))\|_{L^2_{\#}(\R_t;L^2_{*}(\La))}\\
	\leq&c(\nu)\|\n\bw_1\|_{C_{\#}(\R_t;L^2_{*}(\La))}\|\n(\bw_1-\bw_2)\|^{\f14}_{L^{2}_{\#}(\R_t;L^2_{*}(\La))}\|\n^2(\bw_2-\bw_1)\|^{\f34}_{L^2_{\#}(\R_t;L^2_{*}(\La))}\\
	&+c(\nu)\|\n(\bw_1-\bw_2)\|_{C_{\#}(\R_t;L^2_{*}(\La))}\|\n\bw_1\|^{\f14}_{L^{2}_{\#}(\R_t;L^2_{*}(\La))}\|\n^2\bw_1\|^{\f34}_{L^2_{\#}(\R_t;L^2_{*}(\La))}\,.
\end{split}
\end{equation}
Now, we set
\begin{equation}\label{}
B_{\delta}=\{\bw\in L^{\infty}_{\#}(\R_t;\mathbb{V})\cap L^2_{\#}(\R_t;\mathbb{V}_2(\La): \|\bw\|_{C_{\#}(\R_t;\mathbb{V}(\La))\cap L^2_{\#}(\R_t;\mathbb{V}_2(\La))}\leq\delta\}\,.
\end{equation}
By assuming that $\bw\,,\bw_1\,,\bw_2\in B_{\delta}$, from \eqref{es-T} and \eqref{es-T1-T2} it follows that
\begin{equation}\label{eq-T-delta}
\begin{split}
	\|\mT(\bw)\|_{C_{\#}(\R_t;\mathbb{V}(\La))\cap L^2_{\#}(\R_t;\mathbb{V}_2(\La))}
	\leq c(\nu)\delta^2+c(\nu)\|g\|_{H^1_{\#}(\R_t)}\,.
\end{split}
\end{equation}
and
\begin{equation}\label{eq-T1-T2-delta}
\begin{split}
	&\|\mT(\bw_1)-\mT(\bw_2)\|_{C_{\#}(\R_t;\mathbb{V}(\La))\cap L^2_{\#}(\R_t;\mathbb{V}_2(\La))}\\
	\leq& c(\nu)\delta\|\bw_1-\bw_2\|_{C_{\#}(\R_t;\mathbb{V}(\La))\cap L^2_{\#}(\R_t;\mathbb{V}_2(\La))}\,.
\end{split}
\end{equation}
We remark that an explicit expression for $ c(\nu)\,$ can be easily obtained by following the above calculations.\par%
Thus, if
\begin{equation}\label{cnug}
c(\nu)\|g\|_{H^1_{\#}(\R_t)}< \f12\delta\,,\quad c(\nu)\delta< \f12\,,
\end{equation}
from the estimates \eqref{eq-T-delta} and \eqref{eq-T1-T2-delta} it follows that $\,\mT\,$ is a contraction map in $B_{\delta}$. Note that \eqref{cnug}
holds if \eqref{cnug2} below holds. Collecting the above facts, we prove Theorem \ref{thm:navstokes}.
\section{Symmetrical rotation pipes and full developed solutions.}\label{sys-case}
In this section, the spatial domain $\La$ is an infinite symmetrical-rotation pipe with the above $L$-periodic shape in the $z$-axial direction. For simplicity, we consider the physical case $\,n=\,2\,.$ This is a particular case of the case considered in the above sections. So we will not repeat obvious adaptation of notation to this particular case. Here we set $(x_1,x_2,z)=\,(x,y,z)\,.$\par%
It would be of interest to extend the result to more general cases.\par%
Symmetrical-rotation is described as follows. Given a positive $L$-periodic function $r(z)\,$, $\,t \in \R\,$, one has
\begin{equation}
\Sigma_z=\{x^2+y^2<r^2(z)\}\,,
\end{equation}
and so
\begin{equation}\label{lambes}
\La=\{(x,y,z):(x,y)\in\Sigma_z\}=\{(x,y,z):x^2+y^2<r^2(z)\,, z \in \R\}\,.
\end{equation}
The Stokes system \eqref{eq-1}, its abstract form \eqref{eq-4}, and the statement of Theorem \ref{thm:ex} remain in force by replacing $\,x=(x_1,x_2)\,$ by $\,(x,y)\,.$\par%
\begin{theorem}\label{thm:ex-bis}
Assume the above symmetrical-rotation picture where, in particular, $\,\La\,$ is defined by \eqref{lambes}. Then the statement of Theorem \ref{thm:ex} still holds by merely replacing notation $\,x=(x_1,x_2)\,$ by notation $\,(x,y)\,.$
\end{theorem}
Our aim is to study this particular case in a more exhaustive way, having in mind the notion of \emph{full developed solution}. For convenience, we describe our solution in terms of cylindrical coordinates $\,(\ro, \te,z)\,$ and the velocity $\,\bv\,$ by the corresponding components $\,\bv=\,(v_\ro, v_\te, v_z)\,$.  We want to prove the following result.
\begin{theorem}\label{thm:ex-bis}
The solution $\,\bv\,$ of the Stokes evolution problem considered in Theorem \ref{thm:ex-bis} is radial symmetric. Furthermore, the component $\,v_\te\,$ vanishes identically.
\end{theorem}
The reader should note the geometrical significance of the second property. It is obviously necessary to give sense to a unique solution.\par%
It is obvious that a rotation of a solution of our problem around the z-axis is still a solution. Hence, the uniqueness of the solution $\bv$ implies that it must be axis-symmetric, that is, independent of $\theta$. Hence, we can write $\bv$ as
\begin{equation}\label{}
\bv=v_r(r,z)\ee_r+v_{\theta}(r,z)\ee_{\theta}+v_z(r,z)\ee_z\,.
\end{equation}
In addition, we want to prove that $v_{\theta}(r,z)=\,0\,$, that is the component $\,\bar{\bv}\,$ of the velocity lying in $\Si_z\,$ is radial: $\bar{\bv}= v_r(r,z)\ee_r$.
\begin{proof}
Rewriting our system in cylindrical coordinates, one gets
\begin{equation}\label{1}
	\p_{t}v_{\theta}-\left(\p_{rr}+\f{\p_r}{r}+\p_{zz}-\f1{r^2}
	\right)v_{\theta}=0\quad \textrm{in}\ \La\,.
\end{equation}
%
%\begin{equation}\label{1}
	%\begin{cases}
		%\f{d}{dt}v_r-(\p_{rr}+\f{\p_r}{r}+\p_{zz}-\f1{r^2})v_r+\p_r\tilde{p}=0& \text{in $\La_{0,L}$}\,,\\
		%\f{d}{dt}v_{\theta}-(\p_{rr}+\f{\p_r}{r}+\p_{zz}-\f1{r^2})v_{\theta}=0& \text{in $\La_{0,L}$}\,,\\
		%\f{d}{dt}v_z-(\p_{rr}+\f{\p_r}{r}+\p_{zz})v_z+\p_z\tilde{p}=\psi(t)& \text{in $\La_{0,L}$}\,,\\
		%\p_rv_r+\f{v_r}{r}+\p_zv_z=0& \text{in $\La_{0,L}$}\,,\\
		%v_r=v_{\theta}=0& \text{on $S_L$}\,,\,,
		%\end{cases}
	%\end{equation}
By multiplying \eqref{1} by $v_{\theta}$ and integrating over $\Sigma_z$, we get
\begin{equation}\label{inteta}
	\f12\,\f{d}{dt} \int_{\Sigma_z}  v^2_{\theta}\, r\, dr d\te +
	\int_{\Sigma_z} (\p_rv_{\theta})^2 \,r\,dr d\te -\int_{\Sigma_z} \p_z(\p_z v_\te)\, v_\te \,r\,dr d\te +\,
	\int_{\Sigma_z} \, \f{v^2_{\theta}}{r^2}\,r\,drd\te=0\,.
\end{equation}
By integration in $(0,\,L)\,$ with respect to $z$, and by taking into account $z-$periodicity, one gets
\begin{equation}\label{intetad}
	\f12\,\f{d}{dt} \int_{\La_{0,L}}  v^2_{\theta}v\, dxdydz +\,
	\int_{\La_{0,L}} (\p_rv_{\theta})^2\, dxdydz +\,\int_{\La_{0,L}} (\p_zv_{\theta})^2\, dxdydz +\,
	\int_{\La_{0,L}} \, \f{v^2_{\theta}}{r^2}\,dxdydz\,=0\,,
\end{equation}
Finally, by integration in $(0,\,T)\,$ with respect to $t\,,$ and by taking time-periodicity into account, one gets
\begin{equation}\label{}
	\int_0^T\int_{\La_{0,L}}\left[(\p_rv_{\theta})^2+(\p_{z}v_{\theta})^2+\f{v^2_{\theta}}{r^2}\right]\,dxdydz dt=0\,,
\end{equation}
which implies that $v_{\theta}=0$.
\end{proof}

\end{document}